\documentclass[final]{ws-ccm}
\usepackage[notcite,notref]{showkeys}
\usepackage{pdfsync}
\usepackage[usenames,dvipsnames]{color}
\usepackage{mathtools}
\usepackage{mathrsfs}

\newcommand{\RR}{{\mathbb R}}

\newcommand{\NN}{{\mathbb N}}

\renewcommand{\theequation}{\thesection .\arabic{equation}}
\newcommand{\be}{\begin{equation}}
\newcommand{\ee}{\end{equation}}
\newcommand{\bea}{\begin{eqnarray}}
\newcommand{\eea}{\end{eqnarray}}
\newcommand{\beann}{\begin{eqnarray*}}
\newcommand{\eeann}{\end{eqnarray*}}

\begin{document}

\markboth{N. Fusco, V. Millot \&  M. Morini} {A quantitative isoperimetric inequality for fractional perimeters}

\title{A QUANTITATIVE ISOPERIMETRIC INEQUALITY\\ 
FOR FRACTIONAL PERIMETERS}

\author{NICOLA FUSCO}
\address{Dipartimento di Matematica e Applicazioni ``R. Caccioppoli''\\
Universit\`a degli Studi di Napoli ``Federico II''\\
80126 Napoli, Italy\\
\emph{\tt{n.fusco@unina.it}}}

\author{VINCENT MILLOT}
\address{Universit\'e Paris Diderot - Paris 7\\
CNRS, UMR 7598 Laboratoire J.L. Lions\\
F-75005 Paris, France\\
\emph{\tt{millot@math.jussieu.fr}}}

\author{MASSIMILIANO MORINI}
\address{Dipartimento di Matematica\\
Universit\`a degli Studi di Parma \\
43124 Parma, Italy\\
\emph{\tt{massimiliano.morini@unipr.it}}}

\maketitle

\begin{abstract}
{\bf Abstract.} Recently Frank \& Seiringer have shown an isoperimetric inequality for nonlocal perimeter functionals arising from Sobolev 
seminorms of fractional order. This isoperimetric inequality  is improved here in a quantitative form. \end{abstract}


\ccode{}

\section{Introduction}                        

Isoperimetric inequalities play a crucial role in many areas of mathematics such as geometry, linear and nonlinear PDEs, or probability theory.   
In the Euclidean setting, it states that among all sets of prescribed measure, balls 
have the least perimeter. More precisely, for any Borel set $E\subset\RR^N$ of finite Lebesgue measure, 
\begin{equation}\label{stdiso}
N|B|^{1/N}|E|^{(N-1)/N}\leq P(E)\,,
\end{equation}
where $B$ denotes the unit ball of $\mathbb{R}^N$ centered at the origin. Here $P(E)$ denotes the distributional perimeter of $E$ which coincides with the $(N-1)$-dimensional 
measure  of $\partial E$ when $E$ has a (piecewise) smooth boundary. It is a well known fact that inequality \eqref{stdiso} is strict unless 
$E$ is a ball. Here the natural framework for studying the isoperimetric inequality is the theory of sets of finite perimeter.  
We briefly recall that a Borel set $E$ of finite Lebesgue measure is said to be of finite perimeter if its characteristic function $\chi_E$ belongs to $BV(\mathbb{R}^N)$, 
and then $P(E)$ is given by the total variation of the distributional derivative of $\chi_E$. Throughout this paper, we shall refer to the monograph~\cite{AFP} for the 
basic properties of sets of finite perimeter. 

The {\it isoperimetric deficit} of a set $E$ of finite perimeter is defined as the scaling and translation invariant quantity 
$$ D(E):=\frac{P(E)-P(B_r)}{P(B_r)} \,,$$
where $B_r:=rB$ is the ball having the same measure as $E$, {\it i.e.}, $r^N|B|=|E|$. By the characterization of the equality cases in \eqref{stdiso}, the 
isoperimetric inequality rewrites $D(E)\geq 0$, and $D(E)=0$ if and only if $E$ is a translation of $B_r$. 
Hence the isoperimetric deficit measures in some sense how far is a set from being ball. Finding a quantitative version of \eqref{stdiso} consists in proving 
that the isoperimetric deficit controls a more usual notion of ``distance from the family of the balls". 
To this aim is introduced the so-called {\it Fraenkel asymmetry} of the set $E$, and it is defined 
by
$$A(E):=\min\left\{\frac{|E\triangle B_r(x)|}{|E|}: x\in\RR^N\,,\;r^N|B|=|E|\right\}\,,$$
where $B_r(x):=x+rB$, and $\triangle$ denotes the symmetric difference between sets. 
Note that asymmetry is also invariant under scaling and translations. We then look for a positive constant $C_N$ depending  only on the dimension, and 
an exponent $\alpha>0$   such that  $A(E)\leq C_N \big( D(E)\big)^\alpha$, which can be rewritten as a  
quantitative form of \eqref{stdiso}, 
$$ P(E)\geq \left(1+\left(\frac{A(E)}{C_N}\right)^{1/\alpha}\right) N|B|^{1/N}|E|^{(N-1)/N} \,.$$
We shall not attempt here to sketch the history of this problem, but simply refer to the recent paper by Fusco, Maggi, and Pratelli \cite{FMP} 
(and references therein) where this inequality has been first proved with the optimal exponent $\alpha=1/2$, and to Figalli, Maggi, and Pratelli \cite{FiMP} 
for anisotropic perimeter functionals (see also~\cite{CL}, and~\cite{M} for a survey). 
\vskip5pt

The main goal of this paper is to prove a quantitative isoperimetric type inequality for nonlocal perimeter functionals arising from Sobolev seminorms 
of fractional order. First, let us introduce what we call the fractional $s$-perimeter of a set. For $s\in(0,1)$ and a Borel set $E\subset \RR^N$, $N\geq1$, we define the fractional $s$-perimeter of $E$ by
$$P_s(E):=\int_E\int_{E^c}\frac{1}{|x-y|^{N+s}}\,dxdy\, .$$
If $P_s(E)<\infty$, we observe that 
\begin{equation}\label{seminorm}
P_s(E)=\frac{1}{2}[\chi_E]^p_{W^{\sigma,p}(\RR^N)}\,,
\end{equation}
 for $p\geq 1$ and $\sigma p=s$, where $[\cdot]_{W^{\sigma,p}(\RR^N)}$ denotes the Gagliardo $W^{\sigma,p}$-seminorm and $\chi_E$ the characteristic function of $E$. The functional $P_s(E)$ can be thought as a $(N-s)$-dimensional perimeter in the sense that $P_s(\lambda E)=\lambda^{N-s}P_s(E)$ for 
 any $\lambda>0$ (compare to the $(N-1)$-homogeneity of the standard perimeter), and $P_s(E)$ can be finite even if the Hausdorff dimension 
 of $\partial E$ is strictly greater than $N-1$ (see {\it e.g.} \cite{RS}). 
It is also immediately checked from the definition that $P_s(E)<\infty$  for any  set $E\subset\RR$ of finite perimeter and finite measure. 

The fractional $s$-perimeter has already been investigated by several authors, specially by Caffarelli, Roquejoffre, and Savin \cite{CRS}  who 
studied regularity for sets of minimal $s$-perimeter (see also \cite{CV}). Besides the fact that fractional Sobolev seminorms are naturally related to fractional 
diffusion processes, one motivation for studying $s$-perimeters appears when we look at the asymptotic  $s\uparrow1$. It turns out that 
$s$-perimeters give an approximation of the standard perimeter, and 
more precisely, it follows from \cite{D} (see also \cite{BBM})  that for any (bounded) set $E$ of finite perimeter, 
\begin{equation}\label{asymps1}
\lim_{s\uparrow1}\,(1-s)P_s(E)= K_N P(E)\,,
\end{equation}
where $K_N$ is a positive constant  depending only on the dimension. Analysis by $\Gamma$-convergence as $s\uparrow1$ of $s$-perimeter functionals can be found in \cite{Ponce}, and \cite{ADM}. Concerning the behavior of $P_s(E)$ as $s\downarrow 0$, we finally mention that 
\begin{equation}\label{asymps0}
\lim_{s\downarrow0}\,sP_s(E)=N|B|\,|E|\,,
\end{equation}
for any set $E$ of finite measure and finite $s$-perimeter for every $s\in(0,1)$, as a consequence of~\cite[Theorem 3]{MS}.
\vskip5pt

An isoperimetric type inequality for $s$-perimeters has been recently proved by  Frank \& Seiringer \cite{FS}, and it states that 
for any Borel set $E\subset \RR^N$ of finite Lebesgue measure, 
\begin{equation}\label{fracisop}
|E|^{(N-s)/N}\leq  C_{N,s} P_s(E) \,,
\end{equation}
for a suitable constant $C_{N,s}$, with equality holding if and only if $E$ is a ball. Actually, inequality \eqref{fracisop} can be deduced from a symmetrization 
result due to Almgren \& Lieb~\cite{AL}, 
and the cases of equality have been determined in \cite{FS}. The constant $C_{N,s}$ is given in \cite[formula (4.2)]{FS},  
 and we notice that $C_{N,s}$ is of order $(1-s)$ as $s\uparrow 1$, and of order $s$ as $s\downarrow 0$ by \eqref{asymps1} and \eqref{asymps0} respectively. 

%

Inequality \eqref{fracisop} is of course equivalent to saying that 
\begin{equation}\label{fracisopbis}
P_s(B_r)\leq P_s(E)
\end{equation}
for any Borel set $E\subset \RR^N$ such that $|E|=|B_r|$.  
In this paper we prove a quantitative version of inequality \eqref{fracisopbis}. 
To this purpose we introduce the following scaling and translation invariant quantity extending the standard isoperimetric deficit to the fractional setting. For a Borel set $E\subset \RR^N$ of finite measure and $B_r$ such that $|E|=|B_r|>0$, we define the {\it $s$-isoperimetric deficit} as 
$$D_s(E):=\frac{P_s(E)-P_s(B_r)}{P_s(B_r)} \,.$$
We have the following result. 

\begin{theorem}\label{main}
Let $N\geq1$ and $s\in(0,1)$. There exists a constant $\mathcal{C}_{N,s}$ depending only on $N$ and $s$ such that  for any Borel set 
$E\subset\RR^N$ with $0<|E|<\infty$, 
\begin{equation}\label{QIPI}
A(E)\leq \mathcal{C}_{N,s} D_s(E)^{s/4}\,.
\end{equation}
\end{theorem}
\vskip5pt

We emphasize that, as in the standard perimeter case,  the exponent appearing in~\eqref{QIPI} does not depend on the dimension. However we strongly suspect that the optimal exponent should be $1/2$ as for the 
classical quantitative isoperimetric inequality (see \cite{FMP, FiMP, CL}).  The 
dependence on $s$ of the constant $\mathcal{C}_{N,s}$ remains unclear since our method does not yield any precise control as $s\uparrow1$ or 
$s\downarrow0$, but some information can be deduced from  \eqref{asymps1} and \eqref{asymps0}. 

We conclude with a few comments on the proof of Theorem \ref{main}. The key tool used here is a local representation due to Caffarelli \& Silvestre \cite{CS} of the $H^{s/2}$-seminorm. It allows us to rewrite the $s$-perimeter $P_s(E)$ as a Dirichlet type energy of a suitable (inhomogeneous) harmonic extension of the characteristic function of $E$ in $\RR^{N+1}_+$ (see Remark~\ref{extensionrem}). With such a representation in hands, we can adapt some symmetrization techniques developed in \cite{FMP, CFMP}.

\section{Preliminary results}             

Throughout the paper, given $s\in(0,1)$, we shall consider functions belonging to the following weighted Sobolev space
$$
\mathcal{W}^{1,2}_s(\RR^{N+1}_+):=\biggl\{u\in W^{1,1}_{\rm loc}(\RR^{N+1}_+):\,\,\int_{\RR^{N+1}_+}z^{1-s}|\nabla u|^2\,dxdz<+\infty\biggr\}\,,
$$
where $\mathbb{R}^{N+1}_+:=\mathbb{R}^N\times (0,+\infty)$, and $\partial \mathbb{R}^{N+1}_+\simeq \mathbb{R}^N$. 
It can be easily checked that each $u\in\mathcal{W}^{1,2}_s(\RR^{N+1}_+)$  has a trace belonging to $L^2_{\rm loc}(\RR^{N})$ that 
we shall denote by $u(\cdot,0)$.

In the proof of Theorem \ref{main}, a key point is given by the following extension lemma which is a consequence of a recent result by Caffarelli \& Silvestre \cite[Formula (3.7)]{CS}, 
and a well known representation of the $H^{s/2}$-seminorm in Fourier space (see {\it e.g.} \cite[Lemma~3.1]{FLS}). Note that for our purposes we restrict ourselves to $s\in(0,1)$, but Lemma~\ref{extension} 
actually holds for any $s\in(0,2)$. 

\begin{lemma}\label{extension}
Let  $s\in(0,1)$. There exists a constant $\gamma_{N,s}>0$ such that for any function $g\in H^{s/2}(\RR^N)$,  
\begin{equation}\label{SDirichenerg}
\int_{\RR^N}\int_{\RR^N}\frac{|g(x)-g(y)|^2}{|x-y|^{N+s}}\,dxdy= \gamma_{N,s} \int_{\RR^{N+1}_+}z^{1-s}|\nabla u|^2\,dxdz\,, 
\end{equation}
where $u$ is the unique solution in $\mathcal{W}^{1,2}_s(\RR^{N+1}_+)$ of 
$$\begin{cases}
{\rm div}(z^{1-s}\nabla u)= 0 & \text{in }\RR^{N+1}_+\,,\\
u=g & \text{on } \RR^N\,.
\end{cases}
$$ 
Moreover, $u$ is explicitely given by the Poisson formula,
\begin{equation}\label{poisson}
u(x,z)=\lambda_{N,s} \int_{\RR^N} \frac{z^s g(y)}{(|x-y|^2+z^2)^{(N+s)/2}}\,dy\,,
\end{equation}
for a constant $\lambda_{N,s}$ only depending on $N$ and $s$, and $u$ can be characterized as the unique minimizer of
$$
\int_{\RR^{N+1}_+}z^{1-s}|\nabla v|^2\,dxdz\,,
$$
among all functions $v\in\mathcal{W}^{1,2}_s(\RR^{N+1}_+)$ satisfying $v(\cdot,0)=g$.
\end{lemma}

\begin{remark}
The constant $\lambda_{N,s}$ in \eqref{poisson} is precisely given by (see {\it e.g.} \cite{Abra})
\begin{equation}\label{value}
\lambda_{N,s}=\left( \int_{\RR^N} \frac{1}{(|y|^2+1)^{(N+s)/2}}\,dy \right)^{-1}=\frac{\Gamma\big((N+s)/2\big)}{\pi^{N/2}\,\Gamma(s/2)}\,,
\end{equation}
where $\Gamma$ is Euler's Gamma function.
\end{remark}

\begin{remark}\label{extensionrem}
As a consequence of Lemma \ref{extension} and \eqref{seminorm}, for any Borel set $E\subset \RR^N$ of finite Lebesgue measure and finite $s$-perimeter, one has 
$$P_s(E)=\frac{\gamma_{N,s}}{2} \int_{\RR^{N+1}_+}z^{1-s}|\nabla u_E|^2\,dxdz\,,$$
where $u_E$ is the unique solution in $\mathcal{W}^{1,2}_s(\RR^{N+1}_+)$ of 
\begin{equation}\label{eqUE}
\begin{cases}
{\rm div}(z^{1-s}\nabla u_E)= 0 & \text{in }\RR^{N+1}_+\,,\\
u_E=\chi_E & \text{on } \RR^N\,.
\end{cases}
\end{equation}
Note that formula \eqref{poisson} yields $u_E\in C^{\infty}(\RR^{N+1}_+)$,  $0\leq u_E\leq 1$, and $u_E(x,z)\to 0$ as $|x|\to\infty$ for every $z>0$. 
In particular, for every $z>0$ and $t>0$, the set $\{u_E(\cdot, z)>t\}$ is bounded in $\mathbb{R}^N$. In addition, it follows from \eqref{poisson}-\eqref{value}
that $u_E(x,z)\to 1$ as $z\downarrow 0$ at every point $x$ of density 1 of $E$, and $u_E(x,z)\to 0$ as $z\downarrow 0$ at every point $x$ of density~0.  
\end{remark}

The proof of Theorem \ref{main} also makes use of symmetric rearrangements, and we need to recall some well known facts. 
For a  measurable function $g:\RR^N\to[0,\infty)$ such that for all $t>0$, 
\begin{equation}\label{livfin}
\mu(t):=\big|\{g>t\}\big|<\infty\,,
\end{equation}
 the {\it symmetric rearrangement} $g^\sharp$ of $g$ is defined as the unique radially symmetric decreasing function on $\RR^N$ satisfying 
$$
\left|\{g^\sharp>t\}\right|=\mu(t)\qquad\text{for all $t>0$}\,.
$$
It is well known that if $g\in W^{1,1}_{\rm loc}(\RR^N)$ then also $g^\sharp\in W^{1,1}_{\rm loc}(\RR^N)$. Moreover 
(see {\it e.g.} \cite[Lemma 3.2 and (3.19)]{CF}) for a.e. $t>0$, 
\begin{equation}\label{derivmu}
\mu^\prime(t)=-\int_{\{g^\sharp=t\}}\frac{1}{|\nabla g^\sharp|}\,d{\mathcal H}^{N-1}\leq-\int_{\{g=t\}}\frac{1}{|\nabla g|}\,d{\mathcal H}^{N-1}\,.
\end{equation}
P\'olya--Szeg\"o Inequality states that the Dirichlet integral of $g$ decreases under symmetric rearrangement, {\it i.e.}, 
\begin{equation}\label{stdPolSze}
\int_{\RR^N} |\nabla g^\sharp|^2\,dx\leq \int_{\RR^N} |\nabla g|^2\,dx\,.
\end{equation}
 The next proposition gives a quantitative version of inequality  \eqref{stdPolSze} in the special case where $g$ is an {\it $N$-symmetric function}, {\it i.e.}, a function symmetric with respect to all coordinate hyperplanes 
(see \cite[Theorem 3]{CFMP} for a similar result). 
\begin{proposition}\label{PSquant}
Let $N\geq1$. There exists a positive constant $C_N$ such that for any  nonnegative, $N$-symmetric function $g\in H^1(\RR^N)$, one has
$$
\int_{\RR^N}|g-g^\sharp|\,dx\leq C_N|{\rm supp}\,g|^{\frac{N+2}{2N}}\left(\int_{\RR^N}|\nabla g|^2\,dx-\int_{\RR^N}|\nabla g^\sharp|^2\,dx\right)^{1/2}\,.
$$
\end{proposition}

\begin{proof}
We assume first that $|{\rm supp}\,g|<+\infty$.
By  H\"older's inequality we estimate for a.e. $t>0$, 
$$\left(\mathcal{H}^{N-1}\left(\{g=t\}\right)\right)^2\leq  \left(\int_{\{g=t\}}|\nabla g|\,d{\mathcal H}^{N-1} \right)\left(\int_{\{g=t\}}\frac{1}{|\nabla g|}\,d{\mathcal H}^{N-1}\right)\,.$$
From the coarea formula, \eqref{derivmu} and the inequality above, we infer that 
\begin{multline}\label{PSquant1}
\int_{\RR^N}|\nabla g|^2\,dx  =\int_0^\infty dt\int_{\{g=t\}}|\nabla g|\,d{\mathcal H}^{N-1} \\
\geq\int_0^\infty\frac{\left({\mathcal H}^{N-1}\left(\{g=t\}\right)\right)^2}{\int_{\{g=t\}}\frac{1}{|\nabla g|}\,d{\mathcal H}^{N-1}}\,dt 
 \geq\int_0^\infty\frac{\left({\mathcal H}^{N-1}\left(\{g=t\}\right)\right)^2}{-\mu^\prime(t)}\,dt\,.
\end{multline}
Since $|\nabla g^\sharp|$ is constant on $\{g^\sharp=t\}$ for a.e. $t>0$, we obtain in the same way, 
\begin{equation}\label{PSquant1sharp}
\int_{\RR^N}|\nabla g^\sharp|^2\,dx  =\int_0^\infty\frac{\left({\mathcal H}^{N-1}\left(\{g^\sharp=t\}\right)\right)^2}{-\mu^\prime(t)}\,dt\,.
\end{equation}
Recalling that $P\big(\{g>t\}\big)={\mathcal H}^{N-1}\big(\{g=t\}\big)$ for a.e. $t>0$, and that $\{g^\sharp>t\}$ is a ball, 
we infer from \eqref{PSquant1}, \eqref{PSquant1sharp}, and  the classical isoperimetric inequality that
\begin{equation}\label{PSquant2}
\begin{split}
\int_{\RR^N}|\nabla g|^2-|\nabla g^\sharp|^2\,dx & \geq \int_0^\infty\frac{P^2\big(\{g>t\}\big)-P^2\big(\{g^\sharp>t\}\big)}{-\mu^\prime(t)}\,dt \\
& \geq 2\int_0^\infty\frac{P\big(\{g>t\}\big)-P\big(\{g^\sharp>t\}\big)}{P\big(\{g^\sharp>t\}\big)}\cdot\frac{P^2\big(\{g^\sharp>t\}\big)}{-\mu^\prime(t)}\,dt \,.
\end{split}
\end{equation}
Assume now that  $N\geq2$.  Since $\{g>t\}$ is an $N$-symmetric set, and $\{g^\sharp>t\}$ is the ball with the same measure centered at the origin,  
the quantitative isoperimetric inequality proved in \cite{FMP} and Lemma~\ref{tre} below yield
\begin{equation}\label{isopquantP1}
\frac{P\big(\{g>t\}\big)-P\big(\{g^\sharp>t\}\big)}{P\big(\{g^\sharp>t\}\big)}\geq  CA\left(\{g>t\}\right)^2\geq 
\frac{ C}{9}\left(\frac{\left|\{g>t\}\triangle\{g^\sharp>t\}\right|}{\left|\{g^\sharp>t\}\right|}\right)^2\,,
\end{equation}
where ${ C}$ denotes a positive constant depending only on $N$.
Next we notice that \eqref{isopquantP1} is trivially true for $N=1$. 

Observing that for $N\geq 1$, 
$$P\big(\{g^\sharp>t\}\big)=N|B|^{1/N}\left|\{g^\sharp>t\}\right|^{\frac{N-1}{N}}=N|B|^{1/N}\mu(t)^{\frac{N-1}{N}} \quad\text{for all $0<t<{\rm ess}\sup g$}\,,$$ 
and, since $\mu$ is decreasing, 
$$ \int_0^\infty\mu(t)^{2/N}(-\mu^\prime(t))\,dt \leq \frac{N}{N+2} \, |{\rm supp}\,g|^{\frac{N+2}{N}}\,,$$
we infer from \eqref{PSquant2} and \eqref{isopquantP1} that 
\begin{align}
\nonumber\int_{\RR^N}|\nabla g|^2-|\nabla g^\sharp|^2\,dx & \geq C \int_0^\infty\frac{\big|\{g>t\}\triangle\{g^\sharp>t\}\big|^2}{\mu(t)^{2/N}(-\mu^\prime(t))}\,dt \\
\nonumber\displaystyle & \geq \frac{C}{\int_0^\infty\mu(t)^{2/N}(-\mu^\prime(t))\,dt}\left(\int_0^\infty\big|\{g>t\}\triangle\{g^\sharp>t\}\big|\,dt\right)^2 \\
\label{lastineq} \displaystyle & \geq \frac{C}{|{\rm supp}\,g|^{\frac{N+2}{N}}}\left(\int_0^\infty\big|\{g>t\}\triangle\{g^\sharp>t\}\big|\,dt\right)^2\,,
\end{align}
where we have used Jensen's inequality, and $C$ still denotes a positive constant depending only on $N$, possibly changing from line to line.

Finally we estimate 
\begin{multline}\label{estil1diff}
\int_{\RR^N}|g-g^\sharp|\,dx  =\int_{\RR^N}\left|\int_0^\infty\chi_{\{g>t\}}(x)-\chi_{\{g^\sharp>t\}}(x)\,dt\right|\,dx \\
 \leq\int_0^\infty\!dt\!\int_{\RR^N}\left|\chi_{\{g>t\}}(x)-\chi_{\{g^\sharp>t\}}(x)\right|dx=\int_0^\infty\!\left|\{g>t\}\triangle\{g^\sharp>t\}\right|dt\,,
\end{multline}
and the conclusion follows gathering \eqref{lastineq} and \eqref{estil1diff}. 

If $|{\rm supp}\,g|=\infty$, for  $\varepsilon>0$ we set $g_\varepsilon=\max\{g,\varepsilon\}-\varepsilon$. Then, by the first part of the proof we have
\begin{align*}
\int_{\RR^N}|g_\varepsilon-g^\sharp_\varepsilon|\,dx &\leq C_N|{\rm supp}\,g_\varepsilon|^{\frac{N+2}{2N}}\left(\int_{\RR^N}|\nabla g_\varepsilon|^2\,dx-\int_{\RR^N}|\nabla g_\varepsilon^\sharp|^2\,dx\right)^{1/2} \\
& \leq C_N|{\rm supp}\,g_\varepsilon|^{\frac{N+2}{2N}}\left(\int_{\RR^N}|\nabla g|^2\,dx-\int_{\RR^N}|\nabla g^\sharp|^2\,dx\right)^{1/2},
\end{align*}
and the conclusion follows by letting $\varepsilon\to0$.
\end{proof}

In the proof of Proposition \ref{PSquant}, we have used the following simple lemma which is proved in \cite[Lemma 5.2]{M}.
\begin{lemma}\label{tre}
Let $E\subset \mathbb{R}^N$ be an $N$-symmetric Borel set of finite Lebesgue measure, with $|E|=|B_r|$. Then, 
$$
A(E)\leq\frac{|E\triangle B_r|}{|B_r|}\leq 3A(E)\,.
$$
\end{lemma}
\vskip5pt

We continue by showing that the Dirichlet type energy in \eqref{SDirichenerg} decreases under ``horizontal" symmetric rearrangement. 
More precisely, we have the following result. 

\begin{lemma}\label{Morini}
Let $s\in(0,1)$ and $u\in \mathcal{W}^{1,2}_s(\RR^{N+1}_+)$ be a nonnegative function such that $u(\cdot,z)$ is measurable and satisfies \eqref{livfin} for every $z\in(0,\infty)\setminus N$ 
for a (possibly empty) set $N$ of vanishing Lebesgue measure. Let $u^*:\RR^{N+1}_+\to [0,\infty)$ be the function defined by  
$$
u^*(x,z):=\left(u(\cdot,z)\right)^\sharp(x)\qquad \text{ for every $z\in (0,+\infty)\setminus N$ and $x\in\RR^N$}.
$$
Then  $u^*\in \mathcal{W}^{1,2}_s(\RR^{N+1}_+)$,
\begin{equation}\label{VertMorini}
\int_{\RR^{N+1}_+}z^{1-s}|\partial_z u|^2 \,dxdz\geq  \int_{\RR^{N+1}_+}z^{1-s}|\partial_z u^*|^2 \,dxdz\,,
\end{equation}
and 
\begin{equation}\label{HorizMorini}
\int_{\RR^{N+1}_+}z^{1-s}|\nabla_x u|^2 \,dxdz\geq  \int_{\RR^{N+1}_+}z^{1-s}|\nabla_x u^*|^2 \,dxdz\,.
\end{equation}
\end{lemma}
\begin{proof}
First, observe that inequality \eqref{HorizMorini} immediately  follows by applying P\'olya--Szeg\"o inequality to each function $u(\cdot,z)$.

To prove \eqref{VertMorini} we need to recall that, given a nonnegative measurable function $g:\RR^N\to[0,\infty)$ satisfying 
$|\{g>t\}|<\infty$ for all $t>0$ and $\nu\in\mathbb{S}^{N-1}$, the Steiner rearrangement of $g$ in the direction $\nu$ is the unique function $g^\nu$ 
such that $\{g^\nu>t\}$ is the Steiner symmetral in the direction $\nu$ of $\{g>t\}$ for all $t>0$. In turn, the Steiner symmetral $E^\nu$ in the direction $\nu$ 
of $E\subset\mathbb{R}^N$ is defined as follows. Assume for simplicity that $\nu=e_N$, and write $x\in\mathbb{R}^N$ as $x=(x',x_N)$ with  $x'\in \mathbb{R}^{N-1}$. 
Set for  $x'\in \mathbb{R}^{N-1}$, 
$$ E_{x'}:=\left\{t\in\mathbb{R}:\,(x',t)\in E\right\}\,,\quad \ell(x'):=\mathcal{L}^1(E_{x'})\,,$$
where $\mathcal{L}^1$ denotes the outer Lebesgue measure in $\mathbb{R}$, and 
$$\pi(E)^+:=\left\{x'\in \mathbb{R}^{N-1}:\,\ell(x')>0\right\}\,.$$
Then the symmetrized set $E^{e_N}$ is defined by 
$$E^{e_N}:=\left\{x\in\mathbb{R}^N:\, x'\in\pi(E)^+\,,\, |x_N|\leq \ell(x')/2 \right\}\,. $$

\vskip3pt

Let $u\in \mathcal{W}^{1,2}_s(\RR_+^{N+1})$  be a nonnegative function such that 
$u(\cdot, z)\in C^{\infty}_c(\RR^{N})$ for a.e. $z>0$. Given a sequence of directions $\{\nu_k\}\subset \mathbb{S}^{N-1}\times\{0\}$ 
dense in $\mathbb{S}^{N-1}\times\{0\}$, we define by induction the following sequence of iterated Steiner rearrangements:
$$
u_1:=u^{\nu_1},\qquad u_{k+1}:=(u_k)^{\nu_{k+1}}\,.
$$
From the P\'olya--Szeg\"o inequality for Steiner symmetrization, we infer that the sequence $\{u_k\}$ is equibounded in 
$W^{1,2}_{\rm loc}(\RR^{N+1}_+)$ and that for a.e. $z>0$,  
\begin{equation}\label{fettep>1}
 \text{the sequence  $\{u_k(\cdot, z)\}$ is equibounded in 
$W^{1,p}(\RR^N)$ for all $p\geq1$. }
\end{equation}
Therefore, up to a (not relabeled) subsequence, $u_k$ converges weakly in $W^{1,2}_{\rm loc}(\RR^{N+1}_+)$ to a function $v$ which is symmetric with respect 
to all directions $\nu_k$. From \eqref{fettep>1} we  have that for a.e. $z>0$, $v(\cdot, z)\in W^{1,p}(\RR^N)$ for all $p\geq1$. Hence,   by continuity, for all such $z$ it turns out that $v(\cdot, z)$ is symmetric with respect to all direction $\nu\in\mathbb{S}^{N-1}\times\{0\}$. By construction we have 
$$
\big|\{x\in\RR^N:\,u_k(x,z)>t\}\big|=\big|\{x\in\RR^N:\,u(x,z)>t\}\big|\,,\quad\text{for all $k\in\mathbb{N}$, $z\in\RR$, $t>0$}\,,
$$
which yields 
$$
\big|\{x\in\RR^N:\,v(x,z)>t\}\big|=\big|\{x\in\RR^N:\,u(x,z)>t\}\big|\,,\quad\text{for a.e. $z>0$, $t>0$}\,.
$$
Hence $v(\cdot,z)=(u(\cdot,z))^\sharp$ for a.e. $z>0$. Since (see {\it e.g.} \cite[Theorem 1]{Brocco}) for all $k\in\mathbb{N}$
$$
\int_{\RR^{N+1}_+}z^{1-s}|\partial_zu|^2\,dxdz\geq\int_{\RR^{N+1}_+}z^{1-s}|\partial_zu_k|^2\,dxdz\,,
$$
we deduce \eqref{VertMorini} by lower semicontinuity, letting $k\to\infty$ in the above inequality. The general case follows by approximating any nonnegative $u\in \mathcal{W}^{1,2}_s(\RR^{N+1}_+)$ as in the statement of the lemma with 
a sequence $\{u_n\}\subset \mathcal{W}^{1,2}_s(\RR^{N+1}_+)$ of  nonnegative functions such that 
$u_n(\cdot, z)\in C_c^{\infty}(\RR^N)$ for a.e. $z>0$, $u_n\to u$ in 
$W^{1,2}_{\rm loc}(\RR^{N+1}_+)$, and
$$
\int_{\RR^{N+1}_+}z^{1-s}|\partial_zu_n|^2\,dxdz\to \int_{\RR^{N+1}_+}z^{1-s}|\partial_zu|^2\,dxdz
$$
as $n\to\infty$.
\end{proof}

Applying the symmetrization procedure of Lemma~\ref{Morini} to the function $u_E$ defined by~\eqref{eqUE}, 
we find that $u^*_E\in \mathcal{W}^{1,2}_s(\RR^{N+1}_+)$, and that the exceptional set $N$ is  empty since $\{u_E(\cdot,z)>t\}$ is bounded for every $t>0$ and every $z>0$ by Remark~\ref{extensionrem}.  
We now check that the trace of $u^*_E$ on $\mathbb{R}^N$ coincides with the characteristic function of the symmetrized set. 

\begin{lemma}\label{trace}
For any Borel set $E\subset\RR^N$ of finite Lebesgue measure, $u^*_E\in \mathcal{W}^{1,2}_s(\RR^{N+1}_+)$ and $u_E^*(\cdot,0)=\chi_{B_r}$ with $r^N|B|=|E|$.  
\end{lemma}

\begin{proof}
The first assertion directly follows from Lemma~\ref{Morini}.
 Fix now $\varepsilon>0$ and $z>\varepsilon$. Then for every $x\in\mathbb{R}^N$ we may estimate  
\begin{multline*}
\left|u_E(x,z)\right|\leq\left|u_E(x,\varepsilon)\right|+\int_\varepsilon^z\left|\partial_zu_E(x,t)\right|\,dt\\
\leq \left|u_E(x,\varepsilon)\right|+\frac{z^{s/2}}{\sqrt{s}}\left(\int_\varepsilon^zt^{1-s}\left|\partial_zu_E(x,t)\right|^2\,dt\right)^{1/2}\,.
\end{multline*}
Letting $\varepsilon\to 0+$ and recalling Remark~\ref{extensionrem}, we deduce 
\begin{equation}\label{as}
\left|u_E(x,z)\right|\leq
\chi_E(x)+\frac{z^{s/2}}{\sqrt{s}}\left(\int_0^{\infty}t^{1-s}\left|\nabla u_E(x,t)\right|^2\,dt\right)^{1/2}
\end{equation}
for a.e. $x\in\RR^N$.
Since the function on the right-hand side of \eqref{as} belongs to $L^2(\RR^N)$, recalling again Remark~\ref{extensionrem},
by the Dominated Convergence Theorem we 
have $u_E(\cdot, z)\to \chi_E$ in $L^2(\RR^N)$ as $z\to 0^+$. 
 Recall now that  the map
$f\mapsto f^\sharp$ is continuous  in  $L^2(\RR^N)$. Hence,  we may conclude that 
$u^*_E(\cdot, z)=(u_E(\cdot, z))^\sharp\to \chi_E^\sharp=\chi_{B_r}$ in $L^2(\RR^N)$ as $z\to 0^+$, which finishes  the proof of the lemma.
\end{proof}

\section{Proof of Theorem \ref{main}}                        

As in the case of the quantitative isoperimetric inequality for the standard perimeter proved in \cite{FMP}, the strategy consists in reducing the proof of \eqref{QIPI} to the case of 
$N$-symmetric sets, {\it i.e.}, sets symmetric with respect to the $N$ coordinate hyperplanes. To this aim, 
we start by proving  the following continuity lemma which is needed in the proof of Proposition~\ref{Nsymmetry}.

\begin{lemma}\label{continuity}
For every $\varepsilon>0$ there exists $\delta>0$ such that if $E\subset \RR^N$ is a Borel set of finite Lebesgue measure satisfying $A(E)\leq3/2$ and 
$D_s(E)\leq \delta$, then $A(E)\leq \varepsilon$.
\end{lemma}

\begin{proof}
We argue by contradiction assuming that there exists a sequence of Borel sets $E_n\subset \RR^N$ such  that $|E_n|=|B|$, $A(E_n)\leq 3/2$, and 
$$D_s(E_n)\to 0 \quad\text{with} \quad A(E_n)\geq \varepsilon\,,$$ 
for some $\varepsilon>0$. We now apply the concentration-compactness Lemma I.1 of \cite{L} to deduce that there exists a (not relabeled) subsequence $\{E_{n}\}$ such that the following three possible cases may occur:
\vskip5pt
\noindent{\it (i)} (up to translations) the sets $\{E_n\}$ have the property that for every $\delta>0$ there exists $R_\delta>0$ such that $|E_n\cap B_{R_\delta}|\geq |B|-\delta$ for all $n$; 
\vskip5pt
\noindent{\it (ii)} for all $R>0$, $\displaystyle\sup_{x\in\RR^N}|E_n\cap B_{R}(x)|\to 0$ as $n\to+\infty$; 
\vskip5pt
\noindent{\it (iii)} there exists $\lambda\in(0,|B|)$ such that for all $\sigma>0$, there exist $n_0\in\NN$, $E^1_n\subset E_n$, and $E^2_n\subset E_n$ satisfying for all $n\geq n_0$, 
$$\begin{cases} 
\big|E_n\setminus(E^1_n\cup E_n^2)\big|<\sigma\,,\quad \big||E_n^1|-\lambda\big|<\sigma \,,\quad \big||E_n^2|-(|B|-\lambda)\big|<\sigma\,,\\
{\rm dist}(E_n^1,E_n^2)\to +\infty\,.
\end{cases}
$$
Notice that though  Lemma I.1 in \cite{L} is stated in a  seemingly  different form, a quick inspection of the proof shows  that it is in fact equivalent to the above statement.
\vskip3pt

We analyse separately the three cases. 
\vskip3pt

\noindent{\it Case (i).} By the compact embedding of $H^{s/2}(\RR^N)$ into $L^1_{\rm loc}(\RR^N)$, up to a subsequence, there exists a set $F$ such that $\chi_{E_n}\to \chi_{F}$ 
in $L^1_{\rm loc}(\RR^N)$. Hence, for every $\delta>0$ there exists $R_\delta$ such that $|F\cap B_{R_\delta}|>|B|-\delta$, and thus $|F|=|B|$. 
By the assumption $D_s(E_n)\to0$ and the lower semicontinuity of the $s$-perimeter, we infer that $D_s(F)=0$, {\it i.e.}, $F$ is a ball of radius $1$ by the characterisation of 
the equality cases in~\eqref{fracisopbis} proved in \cite{FS}.  
Hence $A(E_n)\leq |B|^{-1}|E_n\triangle F|\to 0$, which contradicts $A(E_n)\geq \varepsilon$ for all $n$. 

\vskip3pt

\noindent{\it Case (ii).} We  observe that this case can not occur since the assumption $A(E_n)\leq3/2$ implies that, up to suitable translation of each $E_n$, 
$|E_n\triangle B|\leq 3|B|/2$. In particular we have $|E_n\cap B|\geq |B|/4$ for every $n$.

\vskip3pt

\noindent{\it Case (iii).} Let us fix an arbitrary constant $\eta>0$. We introduce the regularized kernel $K_\eta:\mathbb{R}^N\times\mathbb{R}^N\to [0,\infty)$ defined by
$$K_{\eta}(x,y):=\begin{cases}
\eta^{-(N+s)} & \text{if } |x-y|<\eta \,,\\[4pt]
\displaystyle \frac{1}{|x-y|^{N+s}} & \text{if } \eta\leq |x-y| \leq \eta^{-1}\,,\\[8pt]
0 & \text{otherwise}\,.
\end{cases}
 $$
We observe that 
\begin{align}
\nonumber P_s(E_n)& \geq \int_{E_n}\int_{E_n^c}K_\eta(x,y)\,dxdy\\
\nonumber &\geq \int_{E^1_n}\int_{E_n^c}K_\eta(x,y)\,dxdy +\int_{E^2_n}\int_{E_n^c}K_\eta(x,y)\,dxdy\\
\label{continuity10}& \geq \int_{E^1_n}\int_{(E^1_n)^c}K_\eta(x,y)\,dxdy +\int_{E^2_n}\int_{(E^2_n)^c}K_\eta(x,y)\,dxdy -\mathscr{R}^1_n-\mathscr{R}^2_n\,, 
\end{align}
where for $i=1,2$, 
$$\mathscr{R}^i_n:=\int_{E^i_n}\int_{E_n\setminus E_n^i} K_\eta(x,y)\,dxdy\,.$$
Since $K_\eta(x,y)=0$ whenever $|x-y|>\eta^{-1}$ and ${\rm dist}(E_n^1,E_n^2)\to +\infty$, we have for $n$ large enough
\begin{equation}\label{continuity11}
\mathscr{R}^i_n=\int_{E^i_n}\int_{E_n\setminus (E_n^1\cup E_n^2)} K_\eta(x,y)\,dxdy\leq \frac{|B|\sigma}{\eta^{N+s}}\,. 
\end{equation}
On the other hand, by Lemma A.2 in \cite{FS}, we have 
\begin{multline*}
\int_{E^i_n}\int_{(E^i_n)^c}K_\eta(x,y)\,dxdy = \frac{1}{2}
\int_{\RR^N}\int_{\RR^N}|\chi_{E^i_n}(x)-\chi_{E^i_n}(y)|K_\eta(x,y)\,dxdy \\
\geq \frac{1}{2}\int_{\RR^N}\int_{\RR^N}|\chi_{B_{r^i_n}}(x)-\chi_{B_{r^i_n}}(y)|K_\eta(x,y)\,dxdy= \int_{B_{r^i_n}}\int_{(B_{r^i_n})^c}K_\eta(x,y)\,dxdy\,,
\end{multline*}
where $(r^i_n)^N|B|=|E_n^i|$. From this last inequality, \eqref{continuity10}, \eqref{continuity11}, and the assumption $D_s(E_n)\to0$, letting $n\to+\infty$ and 
then $\sigma\to 0$, we deduce that
\begin{equation}\label{esticonc}
P_s(B)\geq \int_{B_{r^1}}\int_{(B_{r^1})^c}K_\eta(x,y)\,dxdy + \int_{B_{r^2}}\int_{(B_{r^2})^c}K_\eta(x,y)\,dxdy\,,
\end{equation}
where $(r^1)^N|B|=\lambda$ and $(r^2)^N|B|=|B|-\lambda$. 

Finally, letting $\eta\to 0$ in \eqref{esticonc}, we conclude that 
$$P_s(B)\geq P_s(B_{r^1})+ P_s(B_{r^2})=\left[\bigg(\frac{\lambda}{|B|}\bigg)^{(N-s)/N}+\bigg(1-\frac{\lambda}{|B|}\bigg)^{(N-s)/N}\right] P_s(B)\,,$$
which is impossible by strict concavity. 
\vskip3pt

\end{proof}

The following proposition shows that we can reduce the proof of \eqref{QIPI} to the case of $N$-symmetric sets. 
Its proof is almost entirely similar to the proof of Theorem~2.1 in \cite{FMP} except for a few changes indicated below.
 
\begin{proposition}\label{Nsymmetry}
There exists a constant $C_{N,s}>0$, depending only on $N$ and $s$, such that for every Borel set $E\subset \RR^N$ of finite Lebesgue measure 
there is a $N$-symmetric Borel set $F\subset \RR^N$ 
satisfying $|E|=|F|$, $A(E)\leq C_{N,s}A(F)$, and $D_s(F)\leq 2^N D_s(E)$. 
\end{proposition}
\begin{proof}
Without loss of generality, assume that $|E|=|B|$. Given a direction $\nu\in\mathbb{S}^{N-1}$ and $\alpha\in\mathbb{R}$, 
let us set $H_\nu^\pm=\{x\in\RR^N:\,x\cdot\nu\gtrless \alpha\}$ be two half spaces orthogonal to $\nu$ such that 
$|E^\pm_\nu|=|E|/2$, where $E^\pm_\nu:=E\cap H_\nu^\pm$. 
Up to a translation we may assume that $\alpha=0$, {\it i.e.}, $H_\nu=\partial H^+_\nu$ passes through the origin. We also set 
$$
F^+_\nu:=E^+_\nu\cup \mathcal{R}_\nu(E^+_\nu), \quad F^-_\nu:=E^-_\nu\cup \mathcal{R}_\nu(E^-_\nu)\,,
$$
where $\mathcal{R}_\nu:\RR^N\to\RR^N$ denotes the reflection with respect to $H_\nu$. We claim that
\begin{equation}\label{Nsymmetry1}
P_s(E)\geq \frac{P_s(F^+_\nu)+P_s(F^-_\nu)}{2}\,.
\end{equation}
Indeed, let $u_E$ be the function defined in Remark~\ref{extensionrem}. We write $u_E$ as the sum of $\chi_{H^+_\nu\times\mathbb{R}_+}u_E^+$ and 
$\chi_{H^-_\nu\times\mathbb{R}_+}u_E^-$, where 
$$
u_E^\pm(x,z):=
\begin{cases}
u_E(x,z) & \text{if $x\in H^\pm_\nu$}\,, \\
u_E(\mathcal{R}_\nu(x),z) & \text{otherwise}\,.
\end{cases}
$$
It is well known that $u_E^\pm\in\mathcal{W}^{1,2}_s(\RR^{N+1}_+)$, and that 
$$  \int_{\RR^{N+1}_+}z^{1-s}|\nabla u^\pm_E|^2\,dxdz=2 \int_{H^\pm_\nu\times\mathbb{R}_+}z^{1-s}|\nabla u_E|^2\,dxdz\,.$$
Since $u_E^\pm(\cdot,0)=\chi_{F^\pm_\nu}$ we infer from Lemma~\ref{extension} that 
\begin{align*}
 \int_{\RR^{N+1}_+}z^{1-s}|\nabla u_E|^2\,dxdz & =\frac{1}{2}\int_{\RR^{N+1}_+}z^{1-s}\bigl(|\nabla u_E^+|^2+|\nabla u_E^-|^2\bigr)\,dxdz \\
 & \geq\frac{1}{2}\int_{\RR^{N+1}_+}z^{1-s}\bigl(|\nabla u_{F^+_\nu}|^2+|\nabla u_{F^-_\nu}|^2\bigr)\,dxdz\,,
\end{align*}
from which \eqref{Nsymmetry1} follows.

Next we observe that  the case $N=1$ immediately follows from \eqref{Nsymmetry1}. 
In fact, given the $E\subset\RR$ and denoting by $F^1$ and $F^2$ the set obtained by the construction above,  inequality \eqref{Nsymmetry1} yields  
$$
\frac{D_s(F^1)+D_s(F^2)}{2}\leq D_s(E)\,,
$$
while 
$$
A(E)\leq\frac{|E\triangle(-1,1)|}{2}\leq\frac12\Bigl(\frac{|F^1\triangle(-1,1)|}{2}+\frac{|F^2\triangle(-1,1)|}{2}\Bigr)\leq\frac32\bigl(A(F^1)+A(F^2)\bigr)\,,
$$
where the last inequality follows by Lemma~\ref{tre}. Hence the conclusion follows by taking $F^i$ for which $A(F^i)\geq A(E)/3$.
\vskip3pt

For $N\geq2$, we follow the strategy used in \cite{FMP} which is based on the following claim (see \cite[Lemma 2.5]{FMP}):
\par\noindent
{\sl Claim:} \, {\it There exist two constants $C$ and $\delta$, depending only on $N,s$, such that, given $E$ with $|E|=|B|$ and $D_s(E)\leq\delta$, 
and two orthogonal vectors $\nu_1$ and $\nu_2$ in $\mathbb{S}^{N-1}$, one can find $i\in\{1,2\}$ and $j\in\{+,-\}$  with the property that
$$
A(E)\leq CA(F^j_{\nu_i}),\quad D_s(F^j_{\nu_i})\leq2D_s(E)\,.
$$
}
Let us observe that the claim is easily proved when $A(E)\geq3/2$. Indeed, in this case any of the four possible choices $F^j_{\nu_i}$ would work. In fact, given $i\in\{1,2\}$ and $j\in\{+,-\}$, from \eqref{Nsymmetry1} we have that $D_s(F^j_{\nu_i})\leq2D_s(E)$. Moreover,  by the assumption $A(E)\geq3/2$ we have  $|E\cap B(x)|\leq|B|/4$ for all $x\in\RR^N$, and thus  $|F^j_{\nu_i}\cap B(x)|\leq|B|/2$ for all $x\in\RR^N$. Therefore $A(F^j_{\nu_i})\geq1$.

If instead $A(E)\leq3/2$, the proof of the claim follows exactly as the proof of Lemma~2.5 in~\cite{FMP} with the obvious observation that the continuity Lemma~2.3 in \cite{FMP} must be replaced here by our Lemma \ref{continuity}, which holds since $A(E)\leq3/2$ by assumption.

Once the claim above is proved, the argument used in the proof of Theorem~2.1 in~\cite{FMP} can be reproduced here word for word, thus leading to the conclusion.
\end{proof}

\noindent{\bf Proof of Theorem \ref{main}.} 
Without loss of generality we may assume that $|E|=|B|$, and that $E$ has finite $s$-perimeter. Moreover, 
we may also assume that $D_s(E)\leq 1$, and that  $E$ is an $N$-symmetric set  thanks to Proposition \ref{Nsymmetry}  . 

By Lemma \ref{trace} we have  $u^*_E\in\mathcal{W}^{1,2}_s(\RR^{N+1}_+)$ and $u^*_E=\chi_{B}$ on $\RR^N$, 
and we infer from Lemma~\ref{extension} and Remark~\ref{extensionrem} that 
$$P_s(B)=\frac{ \gamma_{N,s}}{2} \int_{\RR^{N+1}_+}z^{1-s}|\nabla u_{B}|^2\,dxdz\leq \frac{ \gamma_{N,s}}{2} \int_{\RR^{N+1}_+}z^{1-s}|\nabla u^*_E|^2\,dxdz\,.$$
From Lemma \ref{Morini} and Fubini's theorem,   we also deduce that  
\begin{equation}\label{main1}
\begin{split}
 \frac{2P_s(B)}{\gamma_{N,s}}D_s(E) & \geq \int_{\RR^{N+1}_+}z^{1-s}|\nabla u_E|^2\,dxdz-\int_{\RR^{N+1}_+}z^{1-s}|\nabla u^*_E|^2\,dxdz\\
 & \geq \int_{\RR^{N+1}_+}z^{1-s}\big(|\nabla_x u_E|^2-|\nabla_x u^*_E|^2\big)\,dxdz\\
&\geq \int_0^\infty z^{1-s}\left(\int_{\RR^N}|\nabla_x u_E|^2-|\nabla_x u^*_E|^2\,dx\right)\,dz\,.
\end{split}
\end{equation}

Let us now set $v_E:=\left(u_E-\frac12\right)^+$. 
It is standard to check that $v_E\in \mathcal{W}^{1,2}_s(\RR^{N+1}_+)$, $v_E(\cdot,0)=\frac{1}{2}\chi_E$, and that 
\begin{equation*}
\nabla v_E= \chi_{\{u_E>1/2\}}\nabla u_E \quad\text{a.e. in~$\mathbb{R}^{N+1}_+$}\,.
\end{equation*}
By Remark~\ref{extensionrem},  $v_E(\cdot, z)$ has compact 
support and belongs to $H^1(\mathbb{R}^N)$ for all $z>0$, and 
\begin{equation}\label{idvN}
\nabla_x v_E(\cdot,z)=\chi_{\{u_E(\cdot,z)>1/2\}}\nabla_x u_E(\cdot,z) \quad\text{a.e. in~$\mathbb{R}^{N}$}\,.
\end{equation}
Then we observe that    $v^*_E=\left(u^*_E-\frac12\right)^+$, 
so that $v^*_E\in \mathcal{W}^{1,2}_s(\RR^{N+1}_+)$, $v^*_E(\cdot,0)=\frac{1}{2}\chi_B$,  and 
\begin{equation*}
\nabla v^*_E= \chi_{\{u^*_E>1/2\}}\nabla u^*_E \quad\text{a.e. in $\mathbb{R}^{N+1}_+$}\,.
\end{equation*}
In addition, 
by the P\'olya-Szeg\"o inequality we have $v^*_E(\cdot,z)\in H^{1}(\mathbb{R}^N)$ for all $z>0$, and 
\begin{equation}\label{idvstarN}
\nabla_x v^*_E(\cdot,z)= \chi_{\{u^*_E(\cdot,z)>1/2\}}\nabla_x u^*_E(\cdot,z) \quad\text{a.e. in $\mathbb{R}^{N}$}\,.
\end{equation}
\vskip3pt
\noindent Squaring both sides of \eqref{as}, and integrating over $\mathbb{R}^N$, we infer
\begin{multline}
\int_{\RR^N}\left|u_E(x,z)\right|^2\,dx
\leq 2|E|+\frac{2z^s}{s}\int_{\RR^{N+1}_+}t^{1-s}\left|\nabla u_E(x,t)\right|^2dxdt \\
\leq 2|B|+\frac{4}{s\gamma_{N,s}} P_s(E)\leq 2|B|+\frac{8P_s(B)}{s\gamma_{N,s}} =:\beta(s)\,, \label{as2}
\end{multline}
where we have used the fact that $D_s(E)\leq1$ in the last inequality. 
As a consequence, for all $z\in(0,1)$ we have by Chebyshev's inequality, 
\begin{equation}\label{estisupp}
\left|{\rm supp}\, v_E(\cdot,z)\right|=\Bigl|\Bigl\{x\in\RR^N:\,u_E(x,z)\geq\frac12\Bigr\}\Bigr|\leq4\beta(s)\,.
\end{equation}
Since the set $E$ is $N$-symmetric, it follows from  \eqref{poisson} that $u_E$ and $v_E$ inherit the same symmetry with respect to $x$. 
Using Proposition~\ref{PSquant} and \eqref{estisupp}, we may now estimate for all $z\in(0,1)$, 
\begin{multline}\label{main2}
\int_{\RR^N}\left|v_E(x,z)-v_E^*(x,z)\right|\,dx \\
\leq c_N \beta(s)^{\frac{N+2}{2N}}\left(\int_{\RR^N}\left|\nabla_xv_E(x,z)\right|^2-\left|\nabla_xv_E^*(x,z)\right|^2\,dx\right)^{1/2},
\end{multline}
for a suitable constant $c_N>0$ depending only on $N$. 

Next we claim that for all $z>0$,  
\begin{equation}\label{estilevelset}
\int_{\{u_E(\cdot,z)=t\}}|\nabla_xu_E(x,z)|\,d{\mathcal H}_x^{N-1}-\int_{\{u_E^*(\cdot,z)=t\}}|\nabla_xu_E^*(x,z)|\,d{\mathcal H}_x^{N-1}\geq0\,\quad\text{for a.e. $t>0$}\,.
\end{equation}
Indeed, given $z>0$, we may argue as in the proof of \eqref{PSquant1} to obtain for a.e. $t>0$, 
\begin{multline*}
\int_{\{u_E(\cdot,z)=t\}}|\nabla_xu_E(x,z)|\,d{\mathcal H}_x^{N-1}\geq 
\frac{\big(\mathcal{H}^{N-1}(\{u_E(\cdot,z)=t\})\big)^2}{\int_{\{u_E(\cdot,z)=t\}}\frac{1}{|\nabla_x u_E(x,z)|}\, d{\mathcal H}_x^{N-1}} \\
\geq \frac{\big(\mathcal{H}^{N-1}(\{u^*_E(\cdot,z)=t\})\big)^2}{\int_{\{u^*_E(\cdot,z)=t\}}\frac{1}{|\nabla_x u^*_E(x,z)|}\, d{\mathcal H}_x^{N-1}} = 
\int_{\{u^*_E(\cdot,z)=t\}}|\nabla_xu^*_E(x,z)|\,d{\mathcal H}_x^{N-1}\,,
\end{multline*}
using \eqref{derivmu}, the standard isoperimetric inequality, and the fact that $|\nabla_x u^*_E(\cdot,z)|$ is constant on $\{u^*_E(\cdot,z)=t\}$.

Then we derive from \eqref{estilevelset}, \eqref{idvN},  \eqref{idvstarN}, and  the coarea formula that for all $z>0$, 
\begin{align}
\nonumber \int_{\RR^N}&\left|\nabla_xu_E(x,z)\right|^2-\left|\nabla_xu_E^*(x,z)\right|^2\,dx 
= \int_{\RR^N}\left|\nabla_xv_E(x,z)\right|^2-\left|\nabla_xv_E^*(x,z)\right|^2\,dx\\
\nonumber & + \int_{0}^{1/2} dt\biggl(\int_{\{u_E(\cdot,z)=t\}}|\nabla_xu_E(x,z)|\,d{\mathcal H}_x^{N-1}-\int_{\{u_E^*(\cdot,z)=t\}}|\nabla_xu_E^*(x,z)|\,d{\mathcal H}_x^{N-1}\biggr) \\
\label{ineqslice}&\geq \int_{\RR^N}\left|\nabla_xv_E(x,z)\right|^2-\left|\nabla_xv_E^*(x,z)\right|^2\,dx\,.
\end{align}
By an argument similar to the one used in the proof of \eqref{as} and \eqref{as2}, we may estimate for all $z>0$, 
\begin{multline*}
\frac{1}{2}\int_{B}\left|u_E(x,0)-u_E^*(x,0)\right|\,dx=\int_{B}\left|v_E(x,0)-v_E^*(x,0)\right|\,dx\\
\leq \int_{B}\left|v_E(x,z)-v_E^*(x,z)\right|\,dx
+\frac{|B|^{1/2}z^{s/2}}{\sqrt{s}}\biggl(\int_{B\times\RR_+}t^{1-s}|\nabla (v_E- v_E^*)|^2\,dxdt\biggr)^{\frac{1}{2}}\,.
\end{multline*}
From the above inequality,  \eqref{main2}, and \eqref{ineqslice} we deduce that for all $z\in(0,1)$, 
\begin{align*}
|B\setminus  E|
&=\int_{B}\left|u_E(x,0)-u_E^*(x,0)\right|\,dx\\
&\leq 2c_N \beta(s)^{\frac{N+2}{2N}} \left(\int_{\RR^N}\left|\nabla_xv_E(x,z)\right|^2-\left|\nabla_xv_E^*(x,z)\right|^2\,dx\right)^{1/2} \\
 &\qquad+\frac{2|B|^{1/2}z^{s/2}}{\sqrt{s}}\left(\int_{\RR^{N+1}_+}t^{1-s}\left|\nabla v_E(x,t)-\nabla v_E^*(x,t)\right|^2dxdt\right)^{1/2}\\
 &\leq2c_N \beta(s)^{\frac{N+2}{2N}} \left(\int_{\RR^N}\left|\nabla_xu_E(x,z)\right|^2-\left|\nabla_xu_E^*(x,z)\right|^2\,dx\right)^{1/2} \\
 &\qquad+\frac{2\sqrt{2}|B|^{1/2} z^{s/2}}{\sqrt{s}}\biggl(\int_{\RR^{N+1}_+}t^{1-s}\left(|\nabla u_E|^2+|\nabla u_E^*|^2\right)dxdt\biggr)^{1/2} 
\end{align*}
Let us fix $\tau\in (0,1]$ to be chosen. Squaring the first and last sides of the inequality above, multiplying by $z^{1-s}$, and integrating in $(0,\tau)$ with respect to $z$ yields 
\begin{multline*}
\frac{|B\setminus E|^2}{2-s}\tau^{2-s}\leq 8c^2_N \beta(s)^{\frac{N+2}{N}}
 \int_0^1 z^{1-s}\left(\int_{\RR^N}\left|\nabla_xu_E(x,z)\right|^2-\left|\nabla_xu_E^*(x,z)\right|^2\,dx\right)\,dz\\
 + \frac{8|B| \tau^2}{s} \int_{\RR^{N+1}_+}z^{1-s}\left(|\nabla u_E|^2+|\nabla u_E^*|^2\right)dxdz\,.
\end{multline*}
Using the P\'olya--Szeg\"o inequality, the  assumption $D_s(E)\leq1$, and \eqref{main1}, we derive that 
\begin{align}
\nonumber |B\setminus E|^2 & \leq 32c^2_N \beta(s)^{\frac{N+2}{N}} \frac{P_s(B)}{\gamma_{N,s}} D_s(E)\tau^{s-2} +\frac{64|B|P_s(B)}{s\gamma_{N,s}}\tau^s\\
\label{presqfini} &\leq C^*_{N,s} \left( D_s(E) \,\frac{\tau^{s-2}}{2-s} + \frac{\tau^s}{s}\right)\,,
\end{align}
with 
$$C^*_{N,s}:=  \frac{64 P_s(B)}{\gamma_{N,s}} \max\big\{c^2_N \beta(s)^{\frac{N+2}{N}}, |B| \big\}$$ 
which only depends on $s$ and $N$.  Next we observe that, among all values of $\tau\in(0,1]$, the right handside of 
\eqref{presqfini} is minimized for $\tau=\sqrt{D_s(E)}$. Hence, 
$$|B\setminus E|\leq \left(\frac{2C^*_{N,s}}{s(2-s)}\right)^{1/2} D_s(E)^{s/4}\,.$$
Finally we observe that $2|B\setminus E|=|B\triangle E|\geq |B| A(E)$ since $|E|=|B|$, and the proof is complete. 
\prbox
\vskip20pt



\noindent{\small {\bf \it Acknowledgements.}
This research was partially supported by the ERC
Advanced Grants 2008 \emph{Analytic Techniques for Geometric and
Functional Inequalities}. The research of V.M. was also partially supported by 
the Agence Nationale de la Recherche under Grant ANR-10-JCJC 0106.}


\end{document}